\documentclass[reqno]{amsart}
\usepackage{nicefrac}
\usepackage{amsmath}
\usepackage{amsthm}
\usepackage{xcolor}
\usepackage{hyperref}
\hypersetup{colorlinks   = true,
     citecolor    = teal,
     urlcolor     = teal,
     linkcolor    = teal}
     
\usepackage{mathtools}
\usepackage{comment}
     
\usepackage{soul}

\newtheorem{thm}{Theorem}[section]{\bf }{\it }
\newtheorem{prop}[thm]{Proposition}{\bf }{\it }
{\bf }{\it }
{\bf }{\it }
\newtheorem{conj}[thm]{Conjecture}{\bf }{\it }
\newtheorem{quest}[thm]{Question}{\bf }{\it }
{\bf }{\it }
\newtheorem{cor}[thm]{Corollary}{\bf }{\it }
{\bf }{\it }
\newtheorem{lem}[thm]{Lemma}{\bf }{\it }
{\bf }{\it }
\newtheorem{fct}[thm]{Fact}{\bf }{\it }
\theoremstyle{definition}
\newtheorem{defn}[thm]{Definition}{\bf }{\rm }
\newtheorem{rem}[thm]{Remark}{\bf }{\rm }
{\bf }{\rm }

\def\wha{\widehat{a}}
\def\whb{\widehat{b}}
\def\whc{\widehat{c}}
\def\whd{\widehat{d}}
\def\gcd{\mathrm{gcd}}
\def\rad{\mathrm{rad}}

\def\Z{\mathbb{Z}}
\def\Zi{\mathbb{Z}[\imagi]}
\def\R{\mathbb{R}} 
\def\N{\mathbb{N}} 
\def\mod#1{{\;(\mathrm{mod}\;#1)}}
\def\imagi{\boldsymbol{\mathit i}}
\def\imagj{\boldsymbol{\mathit j}}
\def\imagk{\boldsymbol{\mathit k}}
\def\Tr{\mathrm{Tr}} 

\newcommand{\Hu}{\mathbb{I}}

\usepackage{enumitem}
\setlist[itemize]{leftmargin=2em}
\setlist[enumerate]{leftmargin=2em}

\makeatletter
\newlength{\negph@wd}
\DeclareRobustCommand{\negphantom}[1]{%
	\ifmmode
	\mathpalette\negph@math{#1}%
	\else
	\negph@do{#1}%
	\fi
}
\newcommand{\negph@math}[2]{\negph@do{$\m@th#1#2$}}
\newcommand{\negph@do}[1]{%
	\settowidth{\negph@wd}{#1}%
	\hspace*{-\negph@wd}%
}
\makeatother
\def\Wdiff#1#2{\phantom{#2}\negphantom{#1}}

\allowdisplaybreaks

\begin{document}	
	
\title
{Strong $n$-conjectures over rings of integers}

\author{Rupert H\"olzl}
\address{Rupert~H\"olzl and Sören~Kleine, Fakult\"at f\"ur Informatik,
	Universit\"at der Bundeswehr M\"unchen, 
	Neubiberg, Germany}
\email{r@hoelzl.fr}
\email{soeren.kleine@unibw.de}

\author{Sören Kleine}
\address{Frank~Stephan, Department of Mathematics \& School of Computing,
	National University of Singapore, Singapore 119076,
	Republic of Singapore}
\email{fstephan@comp.nus.edu.sg}
\thanks{F.~Stephan's research was supported by the Singapore Ministry of Education AcRF Tier 2 grant MOE-000538-01 and AcRF Tier 1 grants A-0008454-00-00 and A-0008494-00-00.}

\author{Frank Stephan}

\begin{abstract}
We study diophantine equations of the form 
${a_1 + \ldots + a_n = 0}$ where the $a_i$'s are assumed to be coprime and to satisfy certain subsum conditions. We are interested in the limit superior of the qualities of the admissible solutions of these equations, a question that in the case ${n = 3}$ is closely related to the famous $abc$-conjecture. In a previous article, we studied multiple versions of this problem over the ring of rational integers, summarising known results and proving stronger lower bounds. In this article we extend our work to the rings of the Gaussian integers and the Hurwitz quaternions, where a somewhat different picture emerges. In particular, we establish much stronger lower bounds on qualities than for the rational integers.
\end{abstract}

\maketitle

\section{Introduction} \label{section:notation} 
The $abc$-conjecture \cite{Mas85,Oes88,Wal15} is a well-known open problem in mathematics. It postulates that there is no constant $q > 1$ such that there exist infinitely many triples $(a,b,c)$ of coprime and non-zero integers with $a+b+c=0$ and such that a so-called ``quality'' of $(a,b,c)$ exceeds $q$. The conjecture itself is rather well studied but still unresolved. However, on the way towards partial solutions, various variants of the original problem were formulated and conjectures about the achievable qualities in these cases were made.

In a precursor to the present article~\cite{paper1}, we summarised what was known about these variants, and proved new and stronger results. In this article, we extend this work to more general settings, namely to other rings of integers besides~$\Z$. Before we go into details about the new questions we will investigate, we begin by recalling the basic notations used for the (strong) $n$-conjectures over the rational integers that were studied in the previous article. 
\begin{defn} \label{def:quality_klassisch} 
	Let ${n \in \Z}$ be a non-zero integer. The \emph{radical} $\rad(n)$ of $n$ is defined as the largest square-free positive divisor of $n$. 
	
	Next, let ${a = (a_1,\ldots,a_n) \in \Z^n}$, ${n \ge 3}$ be such that ${a_1, \ldots, a_n \ne 0}$. Then we define the \emph{quality} of $a$ as  
	\[q(a)  = 
	\frac{\log(\max(|a_1|,\ldots,|a_n|))}{\log\,\rad(a_1 
		\cdot \ldots \cdot a_n)}. \]
	
	Finally, for a sequence of $n$-tuples ${A = \{a^{(1)}, a^{(2)},\ldots\}  \subseteq \Z^n}$, ${n \ge 3}$, we let the {\em quality}
	of $A$ be defined as
	$$
	   Q_A = \limsup_{k \rightarrow \infty} q(a^{(k)}). 
	$$
\end{defn} 
In the previous  as well as in the present article we study different conjectures which make predictions on the qualities of certain sets~$A$ of $n$-tuples of integers. Different combinations of the following properties will be used to specify these sets.
\begin{defn}
	Let $a = (a_1, \ldots, a_n) \in \Z^n$ be such that ${a_1, \ldots, a_n \ne 0}$. Let ${F \subseteq \mathbb{N}}$ be a finite set such that ${\min F \geq 3}$ in case that ${F \ne \emptyset}$. We introduce the following predicates on $a$: \begin{enumerate}[leftmargin=3em]
		\item[(Z)] $a_1+\ldots+a_n=0$;
		\item[(S1)] there are no ${b_1,\ldots,b_n \in \{0,1\}}$ and $i,j$ with 
				      ${1\leq i,j\leq n}$ such that ${b_i=0}$ and ${b_j=1}$ and 
				      ${\sum_{k=1}^n b_k \cdot a_k = 0}$;
		\item[(S2)] there are no $b_1,\ldots,b_n \in \{-1,0,1\}$ and $i,j$ with
				$1\leq i,j\leq n$ such that $b_i=0$ and $b_j=1$ and
				$\sum_{k=1}^n b_k \cdot a_k = 0$;
		\item[(G1)] $\gcd(a_1,\ldots,a_n) = 1$;
		\item[(G2)] $\gcd(a_i,a_j) = 1$ for all $i,j$ with ${1 \leq i < j \leq n}$;
		\item[(F)] none of the numbers $a_1,\ldots,a_n$ is a multiple of any number in $F$.
	\end{enumerate}
\end{defn}
The conjectures studied in the previous article~\cite{paper1} typically concerned the sets ${A \subseteq \Z^n}$ defined via property~(Z) and one of the two \emph{subsum conditions}~(S1) and~(S2), together with one of the two \emph{gcd conditions}~(G1) and (G2). The condition~(F) was newly introduced in that article~\cite{paper1}. We now make this more precise. 
\begin{conj}[$n$-conjecture; Browkin and Brzezi\'nski~\cite{BB94}] \label{conj:bb} $\,$ \\ 
	Let ${n \geq 3}$ and let ${A(n) \subseteq \Z^n}$ be the set of $n$-tuples 
	${a = (a_1,\ldots,a_n)}$ satisfying conditions~\textnormal{(Z)}, \textnormal{(S1)} and \textnormal{(G1)}. 
	Then $Q_{A(n)}=2n-5$ for every $n$.  
\end{conj}
It is known due to Browkin and Brzezi\'nski~\cite[Theorem~1]{BB94} that ${Q_{A(n)} \ge 2n - 5}$ for any ${n \ge 3}$. 

Browkin~\cite{Bro00} introduced the following conjecture which he referred to as ``strong $n$\nobreakdash-con\-jec\-ture.'' This is the only conjecture mentioned here which does not feature any subsum condition. 
\begin{conj}[Browkin~\cite{Bro00}] \label{conj:bro} $\,$ \\ 
	Let $n \geq3$ and let $B(n)$ be the set of $n$-tuples 
	${a = (a_1,\ldots,a_n)\in \Z^n}$ such that conditions~\textnormal{(Z)} and \textnormal{(G2)} hold for $a$. 
	Then $Q_{B(n)}<\infty$ for every $n$.
\end{conj} 
The next variant of the $n$-conjecture was introduced by Ramaekers. 
\begin{conj}[Ramaekers~\cite{Ram09}]\label{conj_rae} $\,$ \\ 
	Let $n \geq3$ and let $R(n)$ be the set of $n$-tuples ${a = (a_1,\ldots,a_n)\in \Z^n}$ which satisfy the conditions~\textnormal{(Z)}, \textnormal{(S1)} and \textnormal{(G2)}. 
	Then $Q_{R(n)}=1$ for every $n$. 
\end{conj}

In the previous article~\cite{paper1}, we proved lower bounds for the qualities of the following narrower sets of integers. 
\begin{defn}\label{def:U}  
	Let $n \geq 3$ and let ${F \subseteq \mathbb{N}}$ be a finite set, where ${\min F \geq 3}$ in case that ${F \ne \emptyset}$. 
	We let $U(F,n)$ contain all
	${a = (a_1,\ldots,a_n)\in\Z^n}$ satisfying the conditions~\textnormal{(Z)}, \textnormal{(S2)}, \textnormal{(G2)} and \textnormal{(F)} with the set $F$. 
\end{defn}
Since condition~(S2) is more demanding than the subset sum condition~(S1) due to the fact that it allows negative coefficients as well, and since (G2) implies (G1), the sets $U(F,n)$ are potentially smaller than the ones introduced before. Therefore any quality lower bound for $U(F,n)$ will immediately also establish  a lower bound for the sets $A(n)$, $R(n)$ and $B(n)$. More precisely, we recall the following fact from the precursor article~\cite{paper1}. 
\begin{fct}\label{compare_fact}
	For every ${n \in \mathbb{N}}$ and any $F$ as above we have ${Q_{U(F, n)} \leq Q_{A(n)}}$ as well as ${Q_{U(F, n)} \leq Q_{R(n)} \leq Q_{B(n)}}$.
\end{fct} 
\noindent Furthermore, in the same article, we pointed out that each of the
statements 
\[\text{``}Q_{A(3)}=1\text{''}\qquad \text{``}Q_{B(3)}=1\text{''}\qquad \text{``}Q_{R(3)}=1\text{''}\qquad  \text{``}Q_{U(\emptyset,3)}=1\text{''}\]
is equivalent to the $abc$-conjecture. 

The following two main results of that previous work will prove relevant in the remainder of the present article.
\begin{thm}[Hölzl, Kleine, and Stephan~\cite{paper1}]\label{thm:paper1}\quad \nopagebreak
	\begin{enumerate} 
	\item[(a)] Let $F$ be such that 
	$2,5,10 \notin F$. Then ${Q_{U(F,n)} \geq \nicefrac53}$ for each odd ${n \geq 5}$. 

	\item[(b)] Let $n \geq 6$ and let $F$ be an arbitrary finite set such that ${\min F \geq 3}$ in case that ${F \ne \emptyset}$. Then
	$Q_{U(F,n)} \geq \nicefrac54$.
	\end{enumerate} 
\end{thm}
\noindent In particular, Ramaekers' Conjecture~\ref{conj_rae} is wrong for any ${n \geq 5}$. 

\medskip

In the present article we will generalise the above results to a wider class of rings. For this purpose, we will define analogues of the sets $A(.)$ and $U(.,.)$ for these rings, as well as an appropriate version of $\rad$. Before making the necessary preparations that will allow us to make this fully formal, we provide an informal overview of our most important results here:

\begin{itemize}
	\item For the setting of Gaussian integers~$\Zi$, and under some assumptions about~$F$, we will show $Q_{U(\Zi,F,n)} \ge \nicefrac{10}{3}$  for every~${n \geq 4}$. Achieving this result requires two separate proofs for the cases $n=4$ and~$n \geq 5$.
	
	\item If we weaken our requirements and consider the larger set $A(\Zi,n)$ we can even achieve the stronger lower bound ${Q_{A(\Zi,n)} \geq 4n - 10}$ for $n\geq 3$. 
	
	\medskip
	
	\item For the Hurwitz integers~$\Hu$, if we investigate the case $n=3$ corresponding to the $abc$-conjecture, we can show 
	$Q_{A(\Hu,3)} \geq Q_{U(\Hu,\emptyset,3)} \geq 2$.
	\item For general~$n\geq 5$, and making some weak assumptions about~$F$, we can again show $Q_{U(\Hu,F,n)} \ge \nicefrac{10}{3}$.
	\item However, relaxing to~$A(\Hu,n)$, we can even achieve the very strong lower bound $Q_{A(\Hu,n)} \geq 8n - 20$ for $n\geq 3$.
\end{itemize}

\section{General definitions}

We begin by summarising the general properties that we require of the rings that we will be working in; our specific examples will then be the commutative ring of Gaussian integers in Section~\ref{section:Gaussian} and the non-commutative ring of Hurwitz integers in Section~\ref{section:Hamiltonian}. For more information on the notions introduced here we refer the interested reader to Cohn~\cite{cohn}, Jacobson~\cite[Chapter~3]{jacobson}, Beauregard~\cite{beauregard} and Facchini and Fassina~\cite{fassina}. 
\begin{defn} \label{def:rings} 
	Let $R$ be a (not necessarily commutative) associative ring with a unit element 1. We assume that ${1 \ne 0}$ in $R$, and that $R$ does not contain any zero-divisors; that is, $R$ is an \emph{integral domain}. 
	
	\begin{enumerate}
		\item[(a)] Two elements $a$ and $b$ of $R$ will be called \emph{associated} if $b = uav$, where $u$ and $v$ 
		are units in $R$.
		\item[(b)] 	 An element~$p \in R \setminus \{0\}$ is called \emph{irreducible} if it is a non-unit  and not a product of two non-units. 
		
		\item[(c)] Two elements $a$ and $b$ of $R$ are said to be \emph{right similar} if $R/aR \sim R/bR$ as right
		$R$-modules (\emph{left similarity} is defined analogously). If $R$ is commutative then $a$ and $b$ are right similar if and only if they are associated. 
		
		\item[(d)] A \emph{unique factorisation domain} is an integral
		domain $R$ such that every non-unit ${0 \ne z \in R}$ has a factorisation into irreducible elements and
		any two prime factorisations of a given element are isomorphic in the sense that if we have factorisations 
		\[\begin{array}{r@{\;}c@{\;}l} z & = & p_1 \cdot \ldots \cdot p_s \\ 
			& = & q_1 \cdot \ldots \cdot q_t \end{array}  \] 
		with irreducible elements ${p_1, \ldots, p_s}$ and ${q_1, \ldots, q_t}$, then ${s = t}$ and there exists a permutation $\pi$ on the set of indices $\{1,2, \ldots, s\}$ such that the irreducible elements $p_i$ and $q_{\pi(i)}$ are right similar for all $i$; here right similarity is not an arbitrary choice due to the fact below.
	\end{enumerate}

\end{defn}
\begin{fct}[{Cohn~\cite[Corollary~1]{cohn}}] Two elements in a ring without zero-divisors, as in Definition~\ref{def:rings}, are right similar if and only if they are left similar.
\end{fct}	
\begin{rem}
	If a unique factorisation domain $R$ is additionally assumed to be commutative 
	then it is equivalent for an element $p \in R \setminus \{0\}$ to be irreducible or to be a \emph{prime} element, that is, to satisfy that
	\begin{itemize}
		\item $p$ is not a unit, and
		\item whenever $p$ divides a product $zw$, with ${z, w \in R}$, then $p$ divides at least one of the factors $z$ and $w$.
	\end{itemize}
	Thus we may use the notions ``irreducible'' and ``prime'' element synonymously for commutative unique factorisation domains. However, in a non-commutative ring~$R$ they need not coincide: while it still holds that any prime element $p$ is also irreducible, if an irreducible element $p$ is a right divisor of a product $zw$ in $R$ one can only conclude that $p$ is similar to a divisor of either~$z$ or~$w$. 
\end{rem}
	
Unique factorisation domains are not uncommon, as the following fact shows.
\begin{fct}[{Cohn~\cite[Theorem~6.3]{cohn}}] 
Any free associative algebra over a field is a unique factorisation domain.
\end{fct}

Now we restrict to a smaller class of rings. 
\begin{defn}\label{def:rings2} 
	Let $R$ be a unique factorisation domain and a finitely generated free (left) $\Z$-module. We say that $R$ \emph{is equipped with a norm function} if there is a map~$N \colon R \rightarrow \Z$ such that \begin{itemize}
		\item \Wdiff{$N(z)$}{$N(zw)$}$N(z) \ge 0$ for each ${z \in R}$, 
		\item \Wdiff{$N(z)$}{$N(zw)$}$N(z) = 0$ if and only if ${z = 0}$, 
		\item $N(zw) = N(z) \cdot N(w)$ for all $z,w \in R$. 
	\end{itemize} 
\end{defn}

With this we are ready to introduce the quantities that will be the main object of our investigations. 
\begin{defn} \label{def:quality_neu} 
	Suppose that $R$ 
satisfies the hypotheses from Definition~\ref{def:rings2}. Let ${a = (a_1, \ldots, a_n) \in R^n}$ be such that ${a_i \ne 0}$ for each $i$. Then we define its \emph{quality}~as 
	$$ q_R(a) = \frac{\log(\max(N(a_1), \ldots, N(a_n)))}{\log(\rad(N(a_1 \cdot \ldots \cdot a_n)))}. $$ 
	In this article, the underlying ring $R$ will usually be clear from context and we will thus often write~$q$ instead of~$q_R$ to lighten notation.
	
	For a sequence of $n$-tuples ${A = \{a^{(1)}, a^{(2)},\ldots\}  \subseteq R^n}$, ${n \ge 3}$, the {\em quality}
	of $A$ is defined as
	$$
	Q_{R,A} = \limsup_{k \rightarrow \infty} q_R(a^{(k)}). 
	$$ 
	Again, when the choice of~$R$ is clear, we will often just write $Q_A$.
\end{defn}
\begin{rem}
	The above definition of quality coincides with the one given in Definition~\ref{def:quality_klassisch} if ${R = \Z}$ and if we put ${N(z) = |z|}$ for any rational integer $z$. Note that in our applications in the next sections we will consider integral extensions $R$ of $\Z$ that can be embedded in a natural way into some field extension of the field~$\R$ of real numbers. In particular these ring extensions of $\Z$ come with a natural absolute value function. 
	
	However, the norm map $N$ which we will use will \emph{not} be the absolute value function $|.|$ for these proper extensions $R$ of $\Z$, in fact we will have ${|z| = \sqrt{N(z)}}$ for each ${z \in R}$. This is necessary because in Definition~\ref{def:quality_neu} integer values are required in the denominator, and was first proposed by Elkies~\cite{elkies} and elaborated on by Browkin~\cite[Subsections~3.1--3.2]{Bro00}. Note however that we opted for a different, and arguably simpler, approach to generalize the definition of quality to the setting of more general rings than was used in these two contributions.
\end{rem}

We aim to study the quality of the following types of sets. 
\begin{defn} \label{def:U_allg} 
	Let $n \geq 3$ and let ${F \subseteq \mathbb{N}}$ be a finite set, where ${\min F \geq 3}$ in case that ${F \ne \emptyset}$. Suppose that $R$ is a unique factorisation domain which satisfies the hypotheses from Definition~\ref{def:rings2}. Let ${a = (a_1, \ldots, a_n) \in R^n}$ be such that ${a_i \ne 0}$ for each~$i$. We say that ${a \in A(R,n)}$ if $a$ satisfies the following conditions: 
	\begin{enumerate}[leftmargin=3em]
		\item[(Z)] $a_1+\ldots+a_n=0$,
		\item[(S1)] there are no $b_1,\ldots,b_n \in \{0,1\}$ and $i,j$ with
		$1\leq i,j\leq n$ such that $b_i=0$ and $b_j=1$ and
		$\sum_{k=1}^n b_k \cdot a_k = 0$, 
		\item[(G1)] both the left g.c.d.\ and the right g.c.d.\ of $a_1, \ldots, a_n$ equal $1$.
	\end{enumerate} 
	
	\smallskip
	
	\noindent Similarly, we say that ${a \in U(R,F,n)}$ if the following conditions are satisfied: 
	\begin{enumerate}[leftmargin=3em]
		\item[(Z)] $a_1+\ldots+a_n=0$,
		\item[(S2)] there are no $b_1,\ldots,b_n \in \{-1,0,1\}$ and $i,j$ with
		$1\leq i,j\leq n$ such that $b_i=0$ and $b_j=1$ and
		$\sum_{k=1}^n b_k \cdot a_k = 0$, 
		\item[(G2)] for all $i,j$ with ${1 \leq i < j \leq n}$ we have that both the left g.c.d.\ and the right g.c.d.\ of $a_i$ and $a_j$ equal $1$, 
		\item[(F)] none of $a_1,\ldots,a_n$ is a multiple of any element of $F$.
	\end{enumerate}
	
\end{defn}

Then the meta-question we are interested in is the following.
\begin{quest}
	What can be said about $Q_{U(R, F,n)}$ and $Q_{A(R,n)}$ for various $R$, $F$, and $n$? 
\end{quest}

Condition~(F) with a set  ${F \subseteq \mathbb{N}}$ might not seem particularly meaningful in the larger integral domains studied in this article. Indeed, we will only truly use it in results that make a connection with results from the previous article~\cite{paper1} where the setting was~$\Z$. In all other cases we will have $F=\emptyset$.

\begin{rem}\label{rem:fact} 
	It follows directly from the definitions that ${U(R,F,n) \subseteq A(R,n)}$ for any valid choice of $R$, $F$ and $n$. In particular, we have that 
	\[ Q_{U(R,F,n)} \leq Q_{A(R,n)}. \]
\end{rem}
\medskip

We state an auxiliary result here that holds for any ring $R$ as above, and which will be used in the remainder of the article. \goodbreak
\begin{prop}\label{prop:radical}
	\quad\nopagebreak\begin{enumerate} 
		\item[(a)] Let $n \in \N$ and assume that some factorisation $n=n_1\cdot \ldots \cdot n_\ell$ is given that may consist of prime and non-prime elements of $\N$. Then 
		\[\rad(n) \leq \rad(n_1) \cdot \ldots \cdot \rad(n_\ell) .\]
		\item[(b)] Let $R$ be as in Definition~\ref{def:rings2}, let $z \in R$, and assume that some factorisation $z=z_1\cdot \ldots \cdot z_\ell$ is given that may consist of prime elements, non-prime elements and units of $R$. Then 
		\[\rad(N(z))\leq\rad(N(z_1)) \cdot \ldots \cdot \rad(N(z_\ell)).\]
	\end{enumerate} 
\end{prop}
\begin{proof}\quad\nopagebreak\begin{enumerate} 
		\item[(a)] As every prime factor of $n$ needs to be a prime factor of at least one of the~$n_i$'s, this follows immediately from the definition of $\rad$.	
		
		\item[(b)] Using the multiplicativity of $N$ together with~(a),
		\[\begin{array}{r@{\;}c@{\;}l}
			\rad(N(z))&=&\rad(N(z_1\cdot \ldots \cdot z_\ell))\\
			&=&\rad(N(z_1) \cdot \ldots \cdot N(z_1)) \\
			&\leq& \rad(N(z_1)) \cdot \ldots \cdot \rad(N(z_1)),
		\end{array}\]
		as was to be shown.\qedhere
	\end{enumerate}
\end{proof}

\section{Gaussian integers} \label{section:Gaussian} 
The imaginary quadratic field $K = \mathbb{Q}(\imagi ) = \{ a + b \imagi  \mid a,b \in \mathbb{Q}\}$ generated by adjoining to the field of rational numbers the complex fourth root of unity ${\imagi  = \sqrt{-1}}$ is called the field of \emph{Gaussian numbers}. 
The ring of integers of $K$ given by the set ${\Zi = \{a + b i \mid a,b \in \Z\}}$ is referred to as the \emph{Gaussian integers}, and it will be the playground for the current section. 
\begin{defn}
	For any ${z = a + b \imagi  \in \Zi}$ we define its \emph{norm} to be ${N(z) = a^2 + b^2}$. The \emph{absolute value} of $z$ is given by ${|z| = \sqrt{N(z)} = \sqrt{a^2+b^2}}$. 
\end{defn}
We note that ${\Zi \subseteq \mathbb{Q}(\imagi )}$ can be naturally embedded into the field $\mathbb{C}$ of complex numbers; then the above definitions of the norm and the absolute value coincide with the usual complex norm and  absolute value. We recall some basic properties of the ring~$\Zi$; for the proofs we refer the reader to Ireland and Rosen~\cite[Chapter~4,~§4]{IR90} as well as to Ribenboim~\cite[Section~5.3]{Rib01}. 
\goodbreak
\begin{prop} \label{prop:gauss} \quad \nopagebreak
	\begin{enumerate}
		\item[(a)] The ring $\Zi$ is a {\em Euclidean domain} with respect to the norm map, that is, for $a, b \in \Zi$ with ${b \ne 0}$ there exist ${q, r \in \Zi}$ such that $a = qb + r$ and $N(r) < N(b)$. This implies that it is a unique factorisation domain and a principal ideal domain. 
		\item[(b)] An element ${z \in \Zi}$ is a unit if and only if ${N(z) = 1}$. 
		\item[(c)] The set of units of $\Zi$ is given by ${\Zi^\times = \{ \pm1, \pm \imagi \}}$. 
		\item[(d)] The norm map is multiplicative on $\Zi$, that is,  for all ${z, w \in \Zi}$,
		\[{N(z \cdot w) = N(z) \cdot N(w)}.\]
		\item[(e)] Any rational prime number $q \equiv 3 \mod{4}$ is also prime in $\Zi$. 
		\item[(f)] The rational prime numbers $p \equiv1 \mod{4}$ split into two different prime elements of $\Zi$. More precisely, each such~$p$ can be written as ${p = x^2 + y^2}$ with suitable integers $x$ and $y$ (see, for instance, Bundschuh~\cite[Satz~4.1.1]{bundschuh}). Then ${p = (x+\imagi y) \cdot (x-\imagi y)}$, where both factors are prime elements of norm $p$. 
		\item[(g)] The rational prime number 2 splits as $2 = \imagi  \cdot (1-\imagi )^2 = (-\imagi ) \cdot (1+\imagi )^2$; the two elements $1+\imagi $ and $1-\imagi $ are prime and associated via multiplication by the unit~$\imagi $. We have ${N(1 \pm \imagi ) = 2}$. 
		\item[(h)] Every prime element of $\Zi$ is associated with a prime element that is of one of the three types discussed in items~$(e)$, $(f)$ and~$(g)$.
	\end{enumerate} 
\end{prop}
Our first attempt to obtain results about~$Q_{U(\Zi,F,n)}$ and $Q_{A(R,n)}$ for various~$F$ and~$n$  is by reusing the results for rational integers, and in particular Theorem~\ref{thm:paper1}. To this end we compare the qualities~$q_{\Zi}(a)$ and~$q_\Z(a)$ of an integer $n$-tuple~${a \in \Z^n}$. 
\begin{prop} \label{prop:compare_Gauss} 
	Let $a = (a_1, \ldots, a_n) \in \Z^n$ be such that ${a_1, \ldots, a_n \ne 0}$. Then 
	\[ q_{\Zi}(a) = 2 \cdot q_\Z(a). \]
\end{prop} 
\begin{proof}
	The quality functions $q_\Z$ and $q_{\Zi}$ only differ with respect to the application of the norm function. For any ${a_i \in \Z \subseteq \Zi}$ we have ${N(a_i) = a_i^2}$. Thus, 
	\[\begin{array}{r@{\;}c@{\;}l}
		q_{\Zi}(a) & = & \displaystyle\frac{\log(\max( N(a_1), \ldots, N(a_n)))}{\log(\rad(N(a_1) \cdot \ldots \cdot N(a_n))))} \\[1em] 
		& = & \displaystyle\frac{\log(\max( a_1^2, \ldots, a_n^2))}{\log(\rad(a_1^2 \cdot \ldots \cdot a_n^2)))} \\[1em]  
		& = & \displaystyle\frac{2 \cdot \log(\max( |a_1|, \ldots, |a_n|))}{\log(\rad(a_1 \cdot \ldots \cdot a_n)))} \\[1em] 
		&=& 2 \cdot q_\Z(a), 
	\end{array}\]
	as needed.
\end{proof}
\begin{cor} \label{cor:compare_Gauss} 
	Let ${n \ge 3}$ and let ${F \subseteq \N}$ be a finite set, where ${\min F \geq 3}$ in case that ${F \ne \emptyset}$. Then 
	\[\begin{array}{l@{\;}c@{\;}l@{\;}c@{\;}l}
		Q_{\Zi, U(\Zi,F,n)} &\geq& 2 \cdot Q_{\Z, U(\Zi,F,n) \cap \Z^n} &=& 2 \cdot Q_{\Z, U(\Z,F,n)},\\
		Q_{\Zi, A(\Zi,n)} &\geq& 2 \cdot Q_{\Z, A(\Zi,n) \cap \Z^n} &=& 2 \cdot Q_{\Z, A(\Z,n)}.
	\end{array}\]
\end{cor}
\begin{proof}
	We will prove the first assertion; the proofs of the second are analogous. 
	
	\smallskip
	
	Let ${a = (a_1, \ldots, a_n)}$ be an element of $U(\Zi,F,n)$ all of whose entries are non-zero integers. Then 
	$q_{\Zi}(a) = 2 q_\Z(a)$
	by Proposition~\ref{prop:compare_Gauss}, which establishes the inequality in the above statement. 
	
	\smallskip
	
	Now we show that ${U(\Z,F,n) = U(\Zi,F,n) \cap \Z^n}$ which proves the remaining assertion. 
	
	First, let~${a \in U(\Z,F,n)}$. Properties~(Z) and~(S2) do not depend on the underlying ring $R$, that is, they are satisfied if we consider $a$ as an element of $\Zi^n$. To see that (G2) holds also in the larger ring $\Zi$, suppose that there are ${i, j \in \{1, \ldots, n\}}$ with ${i \ne j}$ and a prime element $p$ of $\Zi$ which divides both $a_i$ and~$a_j$. Then, by~Proposition~\ref{prop:gauss}(d), we have 
	\[ N(a_i) = N(p \cdot b_i) = N(p)\cdot N(b_i) \quad\text{and}\quad  N(a_j) = N(p \cdot b_j) = N(p) \cdot N(b_j)\]
	for suitable ${b_i, b_j \in \Zi}$. Thus both $N(a_i)$ and $N(a_j)$ are divisible by the rational integer~${q = N(p)}$, for which we must have ${q \ne 1}$ by Proposition~\ref{prop:gauss}(b). Recalling that $a_i, a_j \in \Z$, we have $N(a_i)=a_i^2$ and~$N(a_j)=a_j^2$, and thus $a_i$ and~$a_j$ are both divisible by a non-trivial rational divisor of~$q$, contradicting our assumption that they are coprime as rational integers. 	
	
Finally, suppose that an element ${f \in F}$ divides one of the $a_i$ in $\Zi$, that is, there exists an~${x \in \Zi}$ such that ${a_i = f \cdot x}$. Since $\mathbb{Q}$ is closed under inverses and under multiplication we have $x \in \Zi \cap \mathbb{Q} = \Z$, which contradicts ${a \in U(\Z,F,n)}$. Thus, condition~(F) also holds, and ${a \in  U(\Zi,F,n) \cap \Z^n}$.
	
	For the converse inclusion, suppose that ${a \in U(\Zi,F,n) \cap \Z^n}$. Then the entries of $a$ are rational integers summing to zero. As above, condition~(S2) does not depend on the underlying ring $R$. The fact that the entries $a_i$ of $a$ are pairwise coprime in the ring $\Zi$ trivially implies that they are as rational integers, as well. Similarly, condition~(F) holding over the larger ring $\Zi$ implies that it also holds over $\Z$. 
\end{proof}
\goodbreak
We are now ready to state our first result in the realm of Gaussian integers.
\begin{thm} \label{thm:G1} \quad \nopagebreak \begin{enumerate} 
	\item[(a)] For each odd ${n \geq 5}$ and any finite set $F$ which does not contain any divisors of 2, 5 or 10 we have that $Q_{U(\Zi,F,n)} \ge \nicefrac{10}{3}$. 
	\item[(b)] Let ${n \geq 6}$ and let $F$ be an arbitrary finite set such that ${\min F \geq 3}$ in case that ${F \ne \emptyset}$. Then ${Q_{U(\Zi,F,n)} \geq \nicefrac{5}{2}}$. 
	\item[(c)] For any ${n \geq 3}$ we have that ${Q_{A(\Zi,n)} \geq 
	4n - 10}$. 
	\end{enumerate} 
\end{thm} 
\begin{proof}
	Parts~(a) and (b) follow by combining Theorem~\ref{thm:paper1} from the previous article~\cite{paper1} with Corollary~\ref{cor:compare_Gauss}. The last assertion follows from a result of Browkin and Brzeziński \cite[Theorem~1]{BB94} and Corollary~\ref{cor:compare_Gauss}. 
\end{proof}

The following result for the case $n=4$ is inspired by an example in an article of Darmon and Granville~\cite[item (e) on p.~542]{DG95} which they credit to private communication with Noam D.\ Elkies.
\goodbreak
\begin{thm} \label{th:elkies}
	$Q_{A(\Zi,4)} 
	\geq Q_{U(\Zi,\emptyset,4)} 
	\geq \nicefrac{10}{3}$.
\end{thm}
\begin{proof}
	In view of Remark~\ref{rem:fact} it suffices to prove the second inequality.
	
	\smallskip
	A simple calculation shows that, independently of~$x$ and~$y$, for
	\[
		\begin{array}{r@{\;}c@{\;}l@{}r@{\;\;}l@{}l}
			a_{x,y}&:=&\phantom{-}(&x^2&+2& \,\cdot\, x \cdot y\;-\!2y^2),\\
			b_{x,y}&:=&-(&x^2&-2& \,\cdot\, x \cdot y\;-\!2y^2),\\
			c_{x,y}&:=&\phantom{-}(&-x^2&+2\imagi & \,\cdot\, x \cdot y\;-\!2y^2),\\
			d_{x,y}&:=&-(&-x^2&-2\imagi &\,\cdot\, x \cdot y\;-\!2y^2),
		\end{array}
	\]
	we have $a_{x,y}^5 + b_{x,y}^5 +\imagi  \cdot c_{x,y}^5 + \imagi  \cdot d_{x,y}^5=0$, which establishes condition~(Z) for the quadruple~$(a_{x,y}^5,\, b_{x,y}^5,\,\imagi  \cdot c_{x,y}^5,\, \imagi  \cdot d_{x,y}^5)$. Let us write 
	\[\wha_{x,y} :=a_{x,y}^5,\quad \whb_{x,y}:=b_{x,y}^5,\quad\whc_{x,y} := \imagi  \cdot c_{x,y}^5,\quad \text{and} \quad\whd_{x,y}:= \imagi  \cdot d_{x,y}^5\]
	to simplify our notation. Now consider the Pell equation 
	\[x^2-2xy-2y^2 = (x-y)^2-3y^2=1;\tag{$\dagger$} \label{pell_fdjs}\]
	since ${(11,4)}$ is a solution and as $3$ is not a square, it has infinitely many solutions ${(x,y) \in (2\N +1) \times(\N\setminus\{0\})}$ (see, for instance, Bundschuh~\cite[Subsection~4.3.3]{bundschuh}).
	For an arbitrary but fixed such $(x,y)$ we make the following observations, omitting the subscripts to lighten notation:
	\begin{itemize}
		\item {\em We have $b=1$:} This is immediate by \eqref{pell_fdjs} and the definition of~$b$ above.
	
	\smallskip
	
		\item {\em Condition~\textnormal{(G2)} holds; in particular, $\,\wha,\whb,\whc,\whd\neq0$:}
		For a contradiction, assume that some prime element~$p$ exists that is a common factor of some distinct pair of numbers taken from $\{a,b,c,d\}$.
		Note that $a-b=4xy$, that $c-d=\imagi  \cdot 4xy$, and that, by an easy calculation,
		\[
			a \cdot b - c \cdot d = (x^4-8x^2y^2+4y^4)-(x^4+8x^2y^2+4y^4)
			= -16x^2 y^2.
		\]
		Thus $p$~would also have to be a factor of~$x$ or of~$2y$; we claim that it cannot be a factor of both. To see this, observe that any common prime factor~$p$ of~$x$ and of~$2y$ also divides ${(x-y)^2-3y^2}$, contradicting~\eqref{pell_fdjs}. 
		
		But if $p$ only divides exactly one of~$x$ or~$2y$, then it is easy to see that $p$~cannot divide any of $a$, $b$, $c$ or $d$, by inspecting the terms appearing in their definitions. This contradiction implies that no~$p$ as above can exist.
		
		\smallskip
		
		\item {\em Condition~\textnormal{(S2)} holds:} Assume that there is a non-trivial subsum 
		\[b_{\,\wha} \cdot \wha + b_{\,\whb} \cdot \whb + b_{\,\whc} \cdot \whc + b_{\whd} \cdot \whd=0\]
		with coefficients ${b_{\,\wha},\dots,b_{\, \whd} \in \{-1,0,1\}}$.
		As only~$\whc$ and~$\whd$ have an imaginary component, we clearly have ${b_{\,\whc} \ne0}$ if and only if ${b_{\,\whd} \ne 0}$. Moreover, as ${\whc - \whd}$ is not a rational number, these two coefficients must in fact be equal; let us consider all three possible cases how this could occur:
		
		\smallskip
		
		In case that ${b_{\,\whc} = b_{\,\whd}=0}$, we would have ${b_{\,\wha} \cdot \wha + b_{\,\whb} \cdot \whb=0}$. We can exclude the subcase  ${b_{\,\wha}=b_{\,\whb}=0}$, as otherwise our subsum would be trivial after all. In the subcase ${b_{\,\wha}=b_{\,\whb}=1}$ we would be left with 
		\[ \begin{array}{r@{\;}c@{\;}l} 
			\wha + \whb & = & 20 y x^9 - 96 y^5 x^5 + 320 y^9 x \\ 
			            & = & 4xy \cdot (5 x^8 + 80 y^8 - 24 x^4 y^4) \\ 
			            & = & 4xy \cdot ((2x^4 - 6y^4)^2 + x^8 + 44y^8), \end{array} \] 
		which is a product of two positive factors, and thus non-zero. Similarly, if ${b_{\,\wha}=1}$ and ${b_{\,\whb}=-1}$ then we have to consider ${\wha - \whb}$; this is non-zero  because~${a = b + 4xy > b=1}$. 
		The remaining two subcases are excluded by noticing that ${-\wha -\whb = -(\wha + \whb)}$ and ${-\wha + \whb = -(\wha - \whb)}$. 
		
		\smallskip
		
		\noindent Next, in case ${b_{\,\whc} = b_{\,\whd} = 1}$, we distinguish nine subcases. \begin{itemize} 
		\item If ${b_{\,\wha}=b_{\,\whb}=1}$ as well, then we would again have a trivial subsum. 
		\item If ${b_{\,\wha}=0}$ and ${b_{\,\whb}=1}$, then condition~(Z) would imply ${\wha=0}$, which would contradict the previous item. 
		\item A symmetric argument works for the subcase ${b_{\,\wha}=1}$ and ${b_{\,\whb}=0}$. 
		\item If ${b_{\,\wha}=0}$ and ${b_{\,\whb}=-1}$, subtracting ${\wha + \whb + \whc + \whd = 0}$ from ${-\whb + \whc + \whd}$ would yield that ${-\wha - 2\whb = 0}$, thus ${\wha = -2 \whb = -2 b^5 = -2}$, which is  false. 
		\item Similarly, for ${b_{\,\wha}=-1}$ and ${b_{\,\whb}=0}$, we falsely obtain ${2\, \wha = -\whb = -1}$. 
		\item If ${b_{\,\wha}=b_{\,\whb}=-1}$, then adding ${\wha + \whb + \whc + \whd = 0}$ to the subsum yields the statement ${2 (\whc + \whd) = 0}$ which is false since 
		\[ \begin{array}{r@{\;}c@{\;}l@{}l} 
			\whc + \whd & = & &-20yx^9 + 96y^5x^5 - 320y^9x \\ 
			            & = &(&-4xy) \cdot ((2x^4 - 6y^4)^2 + x^8 + 44y^8) \end{array} \] 
	    is strictly negative for all $(x,y)$ that we consider. 
		\item If ${b_{\,\wha}=b_{\,\whb}=0}$ then again ${\whc + \whd = 0}$, which we just excluded.
		\item If ${b_{\,\wha}=1}$ and ${b_{\,\whb}=-1}$, then subtracting ${\wha + \whb + \whc + \whd = 0}$ from the subsum yields ${-2 \whb = 0}$, a contradiction. 
		\item Similarly, if ${b_{\,\wha}=-1}$ and ${b_{\,\whb}=1}$, then subtracting ${\wha + \whb + \whc + \whd = 0}$ from the subsum falsely yields ${-2\, \wha = 0}$. 
		\end{itemize} 
		
		\smallskip
		
		\noindent Finally, the case ${b_{\,\whc} = b_{\,\whd} = -1}$ can be handled symmetrically.
		
		\smallskip
		
		\item {\em We have ${1 = N(b) < N(a)}$ and ${N(c)=N(d) \leq N(a)}$:} For the first claim, since 
		\[{a = b + 4xy > b}\]
		we have ${N(a) = a^2 > N(b) = 1}$.
		For the second claim, easy calculations show
			\[
				\begin{array}{r@{\;}c@{\;}l@{\;}c@{\;}l}
					N(a)&=&x^4& +\; 4x^3y - 8xy^3 &+ 4y\mathrlap{^4\!,}\\
					N(c)=N(d)&=&x^4& +\; 8x^2y^2 &+ 4y\mathrlap{^4\!.}\\
				\end{array}
			\]
		Thus, to establish the inequality, it is sufficient to observe that 
		\[
		 	\begin{array}{rr@{\;}c@{\;}ll}
		 		&x^2-2xy-2y^2\phantom{)} &\mathclap{\stackrel{\eqref{pell_fdjs}}{=}}& 1\\
		 	\Rightarrow&4xy	\cdot (x^2-2xy-2y^2)  &>& 0\\
		 	\Rightarrow&(4x^3y - 8xy^3) - 8x^2y^2\phantom{)}  &>& 0\\
		 	\Rightarrow& N(a)-N(c)=N(a)-N(d) &>& 0.
		 	\end{array}
		\]
		
		\smallskip
		
		\item {\em We have $\rad(N(\,\wha \cdot \whb \cdot \whc \cdot \whd\,))\leq a^3$:} Since $N(\,\whc\,)=N(\whd\,)$ we have 
		\[
			\begin{array}{r@{\;}c@{\;}l}
				\rad(N(\,\wha \cdot \whb \cdot \whc \cdot \whd\,))
				&=&\rad(N(\,\wha\,) \cdot N(\,\whb\,) \cdot N(\,\whc\,) \cdot N(\whd\,))\\[0.125em]
				&=&\rad(N(\,\wha\,) \cdot N(\,\whb\,) \cdot N(\,\whc\,))\\[0.25em]
				&\leq& \rad(N(\,\wha\,)) \cdot \rad(N(\,\whb\,)) \cdot \rad(N(\,\whc\,))\\[0.25em]
				&=& \rad(N(a)^5) \cdot \rad(N(b)^5) \cdot \rad(N(c)^5)\\[0.25em]
				&=& \rad(N(a)) \cdot \rad(1) \cdot \rad(N(c))\\[0.25em]
				&\leq& \rad(a^2) \cdot N(c)\\[0.25em]
				&\leq& a \cdot N(a)\\[0.25em]
				&=& a \cdot a^2. 
			\end{array}
		\] 
	\end{itemize}
	
	\smallskip
	
	Thus, taking the limit superior over all $(x,y)$ as above, we obtain
	\[
	\begin{array}{r@{\;}c@{\;}l}
		Q_{A(\Zi,4)} &\geq& \textstyle\limsup_{(x,y)} q(\,\wha_{x,y},\whb_{x,y},\whc_{x,y},\whd_{x,y})\\[0.5em]
		&=& \textstyle\limsup_{(x,y)} \displaystyle\frac{\log(\max(N(\,\wha_{x,y}),N(\,\whb_{x,y}),N(\,\whc_{x,y}),N(\whd_{x,y})))}{\log(\rad(N(\,\wha_{x,y} \cdot \whb_{x,y} \cdot \whc_{x,y} \cdot \whd_{x,y})))}    \\[1.3em]
		&=& \textstyle\limsup_{(x,y)} \displaystyle\frac{\log(\max(N(a_{x,y}^5),N(b_{x,y}^5),N(\imagi \cdot c_{x,y}^5),N(\imagi \cdot d_{x,y}^5)))}{\log(\rad(N(\,\wha_{x,y} \cdot \whb_{x,y} \cdot \whc_{x,y} \cdot \whd_{x,y})))}    \\[1.3em]
		&\geq& \textstyle\limsup_{(x,y)}\displaystyle\frac{\log( \max(N(a_{x,y}),N(b_{x,y}),N(c_{x,y}),N(d_{x,y}))^5)}{\log (a_{x,y}^3)}\\[1.3em]
		&\geq& \textstyle\limsup_{(x,y)}\displaystyle\frac{\log( N(a_{x,y})^5)}{\log (a_{x,y}^3)}\\[1.3em]
		&\geq& \textstyle\limsup_{(x,y)}\displaystyle\frac{\log( (a_{x,y}^2)^{5})}{\log (a_{x,y}^3)}\\[1.3em]
		
		&\geq& \textstyle\limsup_{(x,y)}  \displaystyle\frac{10}{3}\cdot\frac{\log(a_{x,y})}{\log(a_{x,y})}   \\[1em]
		&=& \nicefrac{10}{3},
	\end{array}
	\]
	as had to be shown.
\end{proof}

\section{Hurwitz integers} \label{section:Hamiltonian} 

Hamilton discovered that it is possible to obtain a skew field extension
over the reals by introducing three square roots
$\imagi$, $\imagj$ and $\imagk$ of $-1$ which do not commute; to be more precise, we have the multiplication rules $\imagi \cdot \imagj = \imagk, \imagi \cdot \imagk = -\imagj,
\imagj \cdot \imagk = \imagi$, and $v \cdot w = -v \cdot w$ for distinct
$v,w \in \{\imagi,\imagj,\imagk\}$. 

The \emph{Hurwitz integers}, introduced by Hurwitz~\cite{Hur19}, are then the
restriction of this structure to numbers of the form
$q + r \imagi + s \imagj + t \imagk$ with either ${q,r,s,t} \in \Z$ or ${q,r,s,t} \in \nicefrac12+\Z$, which form a ring~$\Hu$ with~$1$.
\begin{defn} \label{def:quaternion-norm} 
	Let $z = q + r \imagi + s \imagj + t \imagk \in \Hu$. Then the \emph{conjugate} $\bar{z}$  of $z$ is given as 
	\[ \bar{z} = q - r \imagi - s \imagj - t \imagk, \]
	and the \emph{norm} of $z$ is defined via 
	\[ N(z) = q^2 + r^2 + s^2 + t^2 = z \cdot \bar{z}. \]
	Finally, we define the \emph{trace} of $z$ by 
	\[ \Tr(z) = z + \bar{z} = 2q. \]
\end{defn}
We summarise some well-known properties of~$\Hu$ that are relevant for our purposes; for proofs see, for instance, Hardy and Wright~\cite[Sections~20.6--20.7 \& 20.9]{hardy-wright} and Voight~\cite[Chapter~11]{Voight}. 

\goodbreak
\begin{prop} \label{prop:hamilton} \quad \nopagebreak
	\begin{enumerate}
		\item[(a)] The rational integers are contained in the center of $\Hu$, that is, they commute with all Hurwitz integers. 
		\item[(b)] An element ${z \in \Hu}$ is a unit if and only if ${N(z) = 1}$. 
		\item[(c)] The norm map is multiplicative, that is, ${N(z \cdot w) = N(z) \cdot N(w)}$ for any ${z,w \in \Hu}$. 
		\item[(d)] The set of units of $\Hu$ consists of the $24$ elements given by 
		\[ \pm1, \quad \pm \imagi, \quad \pm \imagj, \quad \pm \imagk, \quad and \quad \frac{1}{2}(\pm1 \pm \imagi \pm \imagj \pm \imagk). \]  
		\item[(e)] Let $a, b \in \Hu$ be such that ${b \ne 0}$. Then there exist ${q, r \in \Hu}$ such that 
		\[ a = bq + r \quad and \quad N(r) < N(b). \] 
		In other words, the ring $\Hu$ is (right) norm Euclidean. In particular, every (right) ideal of $\Hu$ is (right) principal, and $\Hu$ is a unique factorisation domain. 
		\item[(f)] An element ${z \in \Hu}$ is irreducible if and only if ${N(z) \in \Z}$ is a prime number. In particular, no rational prime number $p$ is irreducible in $\Hu$. 
	\end{enumerate} 
\end{prop}
\begin{rem}
	One could alternatively consider the subring 
	\[ \mathbb{L}:=\{ q + r \imagi + s \imagj + t \imagk \mid q,r,s,t \in \Z\} \] 
	of $\Hu$ that is sometimes named after Lipschitz. However, this ring does not share the good arithmetical properties of $\Hu$, such as the Euclidean property; in fact, it is not even a (right) principal ideal domain, see Voight~\cite[Example~11.3.8]{Voight}. 
	
	We do however point out the interesting property that every ${z \in \Hu}$ 
	is associated with a ${z' \in \mathbb{L}}$; see, for instance, Hardy and Wright~\cite[Theorem~371]{hardy-wright}. 
\end{rem} 

In particular, it follows from Proposition~\ref{prop:hamilton}$(e)$ that there is a Euclidean algorithm for computing greatest common left and right divisors. Similarly, versions of Bézout's Lemma hold from both sides: for example, given two elements ${a, b \in \Hu}$ with ${b \ne 0}$ there exist ${u, v \in \Hu}$ such that the greatest common right divisor $\delta$ of $a$ and $b$ can be written as 
$\delta = u a + v b$; again, see Hardy and Wright~\cite[Section~20.8]{hardy-wright}, for instance.

The following auxiliary formulae will prove useful. 
\begin{lem} \label{lem:quadrat}  
	Let either ${q,r,s,t} \in \Z$ or ${q,r,s,t} \in \nicefrac12+\Z$, and let ${z = q + r\imagi + s \imagj + t \imagk}$. Then we have that
	\begin{enumerate} 
	\item[(a)] $z^2 = q^2 - r^2 - s^2 - t^2 + 2q(r\imagi + s \imagj + t \imagk)$, 
	\item[(b)] $-\Wdiff{\imagi}{\imagk}\imagi \cdot z\cdot\imagi\Wdiff{\imagi}{\imagk}  =  q + r\imagi - s \imagj - t \imagk$,
	\item[(c)] $-\Wdiff{\imagj}{\imagk}\imagj \cdot z\cdot\imagj\Wdiff{\imagj}{\imagk}  =  q - r\imagi + s \imagj - t \imagk$,
	\item[(d)]  $-\imagk\cdot z\cdot\imagk  =  q - r\imagi - s \imagj + t \imagk$.
	\end{enumerate} 
\end{lem}
\begin{proof} 
	The first part is immediate by using the conventions ${\imagj \imagi = - \imagi \imagj}$, ${\imagk \imagj = - \imagj \imagk}$, and ${\imagk \imagi = - \imagi \imagk}$ together with the fact that $q$, $r$, $s$, and $t$ all lie in the center of~$\Hu$. 
	
	\smallskip
	
	To simplify the proofs of~(b--d) we use that the map 
	\[\begin{array}{r@{\;}c@{\;}c@{\;}l}
		c_\alpha \colon& \Hu& \rightarrow& \Hu\\
		&w &\mapsto& -\alpha \cdot w \cdot \alpha\\
	\end{array}
	  \] 
	is easily seen to be $\Z$-linear for every choice of ${\alpha \in \{\imagi, \imagj, \imagk\}}$. It is thus sufficient to consider how it acts on inputs 
	${\beta \in \{\imagi, \imagj, \imagk\}}$; namely,
	\[ - \alpha \cdot \beta \cdot \alpha = \begin{cases} 
		+\beta & \text{if } \alpha = \beta, \\
		-\beta & \text{if } \alpha \ne \beta, 
	\end{cases} \] 
	for all ${\alpha, \beta \in \{\imagi, \imagj, \imagk\}}$. 
\end{proof} 

\medskip 
We begin by establishing a result that can be understood as a statement about an analogue of the $abc$-conjecture in the setting of Hurwitz integers. 
\goodbreak
\begin{thm}
	$Q_{A(\Hu,3)} \geq Q_{U(\Hu,\emptyset,3)} \geq 2$.
\end{thm}
\begin{proof}
	In view of Remark~\ref{rem:fact} it suffices to prove the second inequality. 
	
	\smallskip
	
	First note that Lemma~\ref{lem:quadrat}(a) implies that, for every~${\ell \in \N}$,
	\begin{equation} \label{eq:rupert} (2^\ell \imagi + \imagk)^2 =  -2^{2\ell}-1.\tag{$\ddagger$}\end{equation} 
	Next consider triples of the form $(2^{2\ell},1, -2^{2\ell}-1)$ with ${\ell \in \N}$. It is obvious that their entries are pairwise coprime and sum to zero. Moreover, the subsum condition~(S2) from Definition~\ref{def:U_allg} is also easily seen to be satisfied for every~${\ell \in \N}$. 
	
	Using Proposition~\ref{prop:radical} and the multiplicativity of the norm multiple times, we see that for any~$\ell \in \N$ we have
	\[\begin{array}{r@{\;}c@{\;}l}
		q(2^{2\ell},1, -2^{2\ell}-1) & = &\displaystyle \frac{\log(\max(N(2^{2\ell}),N(1),N(-2^{2\ell}-1)))}{\log(\rad(N(2^{2\ell} \cdot 1 \cdot (-2^{2\ell}-1))))}\\[1em]
		& \geq &\displaystyle \frac{\log(N(-2^{2\ell}-1))}{\log(\rad(N(2^{2\ell})) \cdot \rad(N(-2^{2\ell}-1)))}\\[1em]
		& = &\displaystyle \frac{\log((-2^{2\ell}-1)^2)}{\log(2 \cdot \rad(N(-2^{2\ell}-1)))}\\[1em]
		& \stackrel{\eqref{eq:rupert}}{\geq} & \displaystyle\frac{\log(2^{4\ell})}{\log(2 \cdot \rad(N((2^\ell \imagi + \imagk)^2)))}\\[1em]
		& = &\displaystyle \frac{4\ell \cdot \log(2)}{\log(2 \cdot \rad(N(2^\ell \imagi + \imagk)))}\\[1em]
		& = &\displaystyle \frac{4\ell \cdot \log(2)}{\log(2 \cdot \rad(2^{2\ell} + 1))}\\[1em]
		& \geq &\displaystyle \frac{4\ell \cdot \log(2)}{\log(2 \cdot 2^{2\ell + 1})}\\[1em]
		& = &\displaystyle \frac{4\ell \cdot \log(2)}{\log(2^{2\ell + 2})}\\[1em]
		& = &\displaystyle \frac{4\ell}{(2\ell + 2)} \frac{\log(2)}{\log(2)}.\\
	\end{array}\]
	It follows that \[\limsup_{\ell \to \infty} q(2^{2\ell},1, -2^{2\ell}-1) = 2,\] 
	which implies $Q_{U(\Hu,\emptyset,3)} \geq 2$ as required.
\end{proof}

In order to derive results about $Q_{U(\Hu,F,n)}$ and $Q_{A(\Hu,n)}$ for various $F$ and~$n$ from the results over rational integers, we compare $q_{\Hu}(a)$ and $q_\Z(a)$ for an integer $n$\nobreakdash-tuple~${a \in \Z^n}$. The proof of the following proposition is analogous to that of Proposition~\ref{prop:compare_Gauss} since, again, ${N(a) = a^2}$~for any rational integer~$a$. 
\begin{prop} \label{prop:compare_Hamilton} 
	Let $a = (a_1, \ldots, a_n) \in \Z^n$ be such that ${a_1, \ldots, a_n \ne 0}$. Then 
	\[ q_{\Hu}(a) = 2 \cdot q_\Z(a).\]
\end{prop} 

The following result can then be shown by a proof analogous to that of Corollary~\ref{cor:compare_Gauss}; we leave it to the reader to verify that non-commutativity and the distinction between left  and right divisors do not interfere with the arguments used there.
\begin{cor} \label{cor:compare_Hamilton} 
	Let ${n \ge 3}$ and let ${F \subseteq \N}$ be a finite set, where ${\min F \geq 3}$ in case that ${F \ne \emptyset}$. Then 
	\[\begin{array}{l@{\;}c@{\;}l@{\;}c@{\;}l}
		Q_{\Hu, U(\Hu,F,n)} &\geq& 2 \cdot Q_{\Z, U(\Hu,F,n) \cap \Z^n} &=& 2 \cdot Q_{\Z, U(\Z,F,n)},\\
		Q_{\Hu, A(\Hu,n)} &\geq& 2 \cdot Q_{\Z, A(\Hu,n) \cap \Z^n} &=& 2 \cdot Q_{\Z, A(\Z,n)}.
	\end{array}\]
	
\end{cor}
\noindent From this, our first general result for Hurwitz integers follows.
\begin{thm} \label{thm:H1}\quad\nopagebreak\begin{enumerate} 
		\item[(a)] For each odd ${n \geq 5}$ and any finite set $F$ which does not contain any divisors of 2, 5 or 10 we have that $Q_{U(\Hu,F,n)} \ge \nicefrac{10}{3}$. 
		\item[(b)] Let ${n \geq 6}$ and let $F$ be an arbitrary finite set such that ${\min F \geq 3}$ in case that ${F \ne \emptyset}$. Then ${Q_{U(\Hu,F,n)} \geq \nicefrac{5}{2}}$. 
		\item[(c)] For each $n \geq 3$ we have that ${Q_{A(\Hu,n)} \geq 
		4n - 10}$. 
	\end{enumerate} 
\end{thm} 
\begin{proof}
	These statements follow by combining Corollary~\ref{cor:compare_Hamilton} with  Theorem~\ref{thm:paper1} from the previous article~\cite{paper1} and with a result of Browkin and Brzeziński~\cite[Theorem~1]{BB94}, respectively.
\end{proof}

The next result shows that we can achieve better bounds than the ones obtainable from the previously known results for the case of rational integers. 
\begin{thm} \label{thm:Quaternionen-A} 
	For each ${n \geq 3}$ we have that ${Q_{A(\Hu,n)} \geq 
	8n - 20}$.
\end{thm} 
\begin{proof}
	First consider the case ${n = 3}$. Let $a$ and $b$ be natural numbers which satisfy the Pell equation ${a^2-2b^2=1}$. Note that there exist infinitely many solutions of this equation since $2$ is not a square and $(3,2)$ is a solution (see, for instance, Bundschuh~\cite[Subsection~4.3.3]{bundschuh}). Now fix any solution ${(a,b)}$ and consider~${y = a + b \imagi + b \imagj}$. Then Lemma~\ref{lem:quadrat}(a--b) implies  
	\[\begin{array}{r@{\;}c@{\;}l} -\imagi \cdot y^2 \cdot \imagi & = & -\imagi \cdot (a^2 - 2 b^2 + 2ab \imagi + 2ab \imagj) \cdot \imagi\\ 
		& = & -\imagi \cdot (1 + 2ab \imagi + 2ab \imagj) \cdot \imagi \\ 
		& = & 1 + 2ab \imagi - 2ab \imagj. \end{array} \] 
	An easy computation shows that ${x := y^2 \cdot (-\imagi) \cdot y^2 \cdot \imagi}={1 + 4ab \imagi - 8 a^2 b^2 \imagk}$. Furthermore, by Lemma~\ref{lem:quadrat}(c) we have ${-\imagj \cdot x \cdot \imagj = \overline{x}}$ and thus ${x + \overline{x} = \Tr(x) = 2}$.
	Then, for any~$y$ and~$x$ constructed in this way, the triples 
	\[ a(x) = 
	(x, \overline{x}, -2) \] 
	satisfy the zero sum condition~(Z) from Definition~\ref{def:U_allg}. Since $x$ is not a rational integer, it is also clear that there do not exist any non-trivial vanishing subsums. Finally, as ${x = 1 + 2 \cdot z}$ and ${\overline{x} = 1 - 2z}$ for some ${z \in \Hu}$, it is easy to see that $x$, $\overline{x}$~and~$-2$ are pairwise coprime; indeed, any common irreducible divisor $p$ of $x$ and $\overline{x}$~would also have to divide ${x + \overline{x} = 2}$. But then $p$ would divide~${1 = x - 2z}$ and~${1 = \overline{x} + 2z}$, which is impossible. 
	
	To compute the qualities of our triples, note that, by Proposition~\ref{prop:hamilton}(b--c),
	\[ N(x) = N(\overline{x}) = N(y)^4. \]
	Moreover, again using the multiplicativity of the norm, we see that
	\[ \rad(N(x \cdot \overline{x} \cdot 2)) = \rad(N(x)^2 \cdot 4) = \rad(N(y)^8 \cdot 4)\] 
	divides $2 N(y)$. Summarising, we obtain
	\[ q(a) \ge \frac{\log(N(y)^4)}{\log(2 N(y))} = \frac{4 \cdot \log(N(y))}{\log(N(y)) + \log(2)}. \] 
	Plugging different~$x$'s as above into this equation, constructed from their respective~$y$'s with larger and larger~$N(y)$, we see that 
	\[Q_{A(\Hu,3)} \geq \textstyle\limsup_x q(a(x)) = 4,\]
	which happens to equal $8 \cdot 3 - 20$ as required. 
	
	\medskip
	
	For the general case $n > 3$ we proceed with a similar but more complex argument using the same~$y$'s and $x$'s as before but involving higher powers. More precisely, consider the $n$-tuples
	\[
		a(x) := (x^{2n-5},\;\overline x^{\,2n-5},\;-c_{0},\;-c_{1} \cdot N(x),\;
		-c_{2} \cdot N(x)^2,\;\ldots,\;-c_{n-3} \cdot N(x)^{n-3});
	\]
	here the $c_i$'s are fixed coefficients that we will chose below. As before, there exist~$x$'s with arbitrarily high~$N(x)$ constructed from~$y$'s with arbitrarily high~$N(y)$; when we calculate the qualities we can therefore assume that these values are large enough for our purposes. To complete the proof we proceed in the following steps.
	\smallskip
	
	\begin{itemize}
		\item {\em Determining the coefficients $c_0,\ldots, c_{n-3}$:} 
		The equation 
		\begin{equation} \label{eq:rekursion} 
		x^{m+2}+ \overline x^{\,m+2}  =  
		2 \cdot (x^{m+1}+\overline x^{\,m+1}) -
		N(x) \cdot (x^m+\overline x^{\,m}) \tag{$\ast$} 
		\end{equation} 
		can be verified for every $m\in\N$ by an easy calculation. If we use 
		\[\begin{array}{r@{\;}c@{\;}r@{\;}c@{\;}r@{\;}c@{\;}r@{\;}c@{\;}r}
			x^0+\overline x^{\,0} & = & 2\rlap{,} \\
			x^1+\overline x^{\,1} & = & 2\rlap{,} \\
		\end{array}\]
		as induction starts and then apply the above equation inductively, we can generate the following schema of true formulae.
		
			\[\begin{array}{r@{\;}c@{\;}r@{\;}c@{\;}r@{\;}c@{\;}r@{\;}c@{\;}r}
				x^0+\overline x^{\,0} & = & 2\rlap{,} \\
				x^1+\overline x^{\,1} & = & 2\rlap{,} \\
				x^2 + \overline x^{\,2} & = & 4 &-& 2 \cdot N(x)\rlap{,} \\
				x^3 + \overline x^{\,3} & = & 8 &-& 6 \cdot N(x)\rlap{,} \\
				x^4 + \overline x^{\,4} & = & 16 &-& 16 \cdot N(x)&+&2 \cdot N(x)^2\rlap{,} \\
				x^5 + \overline x^{\,5} & = & 32 &-& 40 \cdot N(x)&+&10 \cdot N(x)^2\rlap{,} \\
				x^6 + \overline x^{\,6} & = & 64 &-& 96 \cdot N(x)&+&36 \cdot N(x)^2&-&2\cdot N(x)^3, \\
				x^7 + \overline x^{\,7} & = & 128 &-& 224 \cdot N(x)&+&112 \cdot N(x)^2 &-&14 \cdot N(x)^3.\\[-0.5em]
				&\vdots
			\end{array}\]
			
			An easy induction using~\eqref{eq:rekursion} shows that, for any odd~${m\in \N}$, the sum ${x^m + \overline x^{\, m}}$ can be written as a polynomial in $N(x)$ of degree ${d(m) := (m-1)/2}$ and with integer coefficients, while for even $m$ we obtain a polynomial of degree ${d(m) :=  m/2}$ in $N(x)$. For all $m\geq 0$ and $0 \leq i \leq d(m)$, we denote by ${c_{i,m}}$ the rational integers that occur in the equation 
			\begin{equation} \label{eq:koeff} x^m + \overline x^{\, m} = c_{0,m} + c_{1,m} \cdot N(x)  
			+ \ldots + c_{d(m),m} \cdot N(x)^{d(m)} \tag{$\ast\ast$} \end{equation}
			produced in this way.
			
			\smallskip
			
			Finally, for our given~$n$, we specifically consider $m:=2n-5$ and 
			let $c_i := c_{i,m}$ for all $i \in\{0,\ldots,d(m)\}$, where an easy calculation shows $d(m)=n-3$.
		
		\smallskip
		
		\item {\em Condition~\textnormal{(Z)} clearly holds by the equations above and by the choice of the $c_i$\!'s.}
		
		\smallskip
		
		\item {\em Verifying condition~\textnormal{(S1)}:} We claim that ${c_i > 0}$ whenever $i$ is even, and that ${c_i < 0}$ whenever $i$ is odd, and will prove this as Lemma~\ref{lem:koeff} below; in particular all $c_i$'s are non-zero. 
		
		Suppose we are given some $(b_1,\ldots,b_n)\in\{0,1\}^n$ such that~${\sum_{i=1}^n b_i \cdot a_i = 0}$, for at least one~$i$ we have ${b_i = 0}$, and for at least one~$j$ we have ${b_j = 1}$.
		Since $x^m$~and~$\overline x^{\,m}$ are the only entries in $a(x)$ with an imaginary component, we must have ${b_1=b_2}$. We can w.l.o.g.\ assume~ ${b_1=b_2=0}$; otherwise replace $(b_1,\ldots,b_n)$ by $(1-b_1,\ldots,1-b_n)$ which by~(Z) also has the necessary properties. 
		
		Thus we have 	
		\[ \sum_{i \in I} c_i \cdot N(x)^i = 0\] 
		for some ${\emptyset \neq I \subseteq \{0, 1, \ldots, n-3\}}$. Let ${j = \max_{i \in I} i}$; then	
		\[ \label{eq:invalid} 
			\sum_{i \in I \setminus \{j\}} c_i \cdot N(x)^i = -c_j \cdot N(x)^j. \tag{$+$} 
		\]  
		
		On the other hand, if $x$ has been chosen such that
		\[ N(x) > |c_0| + |c_1| + \ldots + |c_{n-3}|, \]
		then, using that ${|c_j| \geq 1}$ by Lemma~\ref{lem:koeff}, we have
		\[\begin{array}{r@{\;}c@{\;}c@{\;\cdot\;}l@{\;}l}
			\left| \sum_{i \in I \setminus \{j\}} c_i \cdot N(x)^i \right| & \le & \sum_i |c_i|  &N(x)^i \\[0.35em] 
			& \le & (\sum_i |c_i|)  &N(x)^{j-1} \\[0.35em] 
			& < & N(x)  &N(x)^{j-1} \\[0.35em] 
			& \le & |c_j|  &N(x)^j, 
		\end{array} \] 
		a contradiction to equation~\eqref{eq:invalid}.
				
		\smallskip
		
		\item {\em Verifying condition~\textnormal{(G1)}:} We first claim that $c_0=2^m$ for some~$m\in \N$. This is apparent in the equation schema above, and can be formally verified by observing that, for some ${w \in \Hu}$,
		\[ 2^m = (x+\overline{x})^m = x^m + \overline{x}^m + N(x) \cdot w.\] 
		On the other hand, recall that ${x = 1 + 2z}$ for some~${z \in \Hu}$. This is enough to ensure that the first and third components of~$a(x)$ are coprime with each other, which in turn establishes condition~\textnormal{(G1)}.
			
		\smallskip
				
		\item {\em Calculating the qualities:} Without further mention, we will use the multiplicativity of the norm, the fact that ${N(t) = t^2}$ for rational numbers $t$, and that \[{N(x) = N(\overline{x}) = N(y)^4}.\] 
			
		The numerator appearing in $q(a(x))$ is greater than or equal to 
		\[ \log(N(x^{2n-5})) = (2n-5) \log(N(x)) = (2n-5) \cdot 4 \log(N(y)). \] 
		Since the $c_i$'s were chosen as rational integers, the radical that appears in the denominator of $q(a(x))$ is a divisor of 
		$N(y) \cdot c_0 \cdot \ldots \cdot c_{n-3}$ and therefore 
		\[ q(a(x)) \geq \frac{4(2n-5)\cdot\log(N(y))}{\log(N(y)) + \log(c_0 \cdot \ldots \cdot c_{n-3})}. \] 
		Recall that the coefficients $c_i$ do not depend on~$x$ or~$y$. Thus, again plugging in different~$x$'s constructed from~$y$'s with larger and larger~$N(y)$ into this equation, we see that 
		\[Q_{A(\Hu,n)} \geq \textstyle\limsup_x q(a(x)) = 4(2n-5) = 8n-20.\qedhere\]
		
	\end{itemize}
\end{proof}
\noindent We conclude by establishing the claim used during the proof of Theorem~\ref{thm:Quaternionen-A}.
\begin{lem} \label{lem:koeff} 
	Let ${m \geq 1}$ and for $0 \leq i \leq d(m)$ consider the coefficients $c_{i,m}$ as defined in~\eqref{eq:koeff}. We have that ${c_{i,m} > 0}$ whenever $i$ is even, and that ${c_{i,m} < 0}$ whenever $i$ is odd.
\end{lem} 
\begin{proof} 
	We proceed by induction. For ${m = 1}$, where we only have ${c_{0,1} = 2 > 0}$ to consider, the claim is trivially true; similarly for ${m = 2}$ with ${c_{0,2} = 4 > 0}$ and ${c_{1,2} = -2 < 0}$. 
	
	\smallskip
	
	Now suppose that the statement is true for all ${m \le r}$, for some even~${r \geq 2}$. We will show that the assertion also holds for ${m = r+1}$ and ${m = r+2}$.
	
	\begin{itemize}
		\item {\em $m = r+1$:} By the recursion formula~\eqref{eq:rekursion} we have 
		\[\begin{array}{c@{\;}l@{\;}}
			&x^m + \overline x^{\,m}\\[0.35em]
			 = & 2 \cdot (x^r + \overline x^{\,r}) - N(x) \cdot (x^{r-1} + \overline x^{\,{r-1}}) \\[0.35em] 
			 = & 2 \cdot \left(c_{0,r} + c_{1,r} N(x) + \ldots + c_{r/2,r} N(x)^{r/2}\right) \\[0.15em] 
			 & \;\; - N(x) \cdot \left(c_{0,r-1} + c_{1, r-1}N(x) + \ldots + c_{r/2-1,r-1} N(x)^{r/2-1}\right) \\[0.35em] 
			 = & 2 c_{0,r} + N(x) \cdot (2 c_{1,r} - c_{0,r-1}) + N(x)^2 \cdot (2 c_{2,r} - c_{1,r-1})  \\[0.15em] 
			 & \;\; + \ldots + N(x)^{r/2} \cdot (2c_{r/2,r} - c_{r/2-1, r-1}). 
		\end{array} \] 
		It follows that by construction we have 
		\[{c_{0,r+1} = 2 c_{0,r}} \quad\text{and}\quad {c_{i,r+1} = 2 c_{i,r} - c_{i-1,r-1}},\]
		and thus, using the induction hypothesis, it is easy to verify that 
		${c_{0,r+1} > 0}$ and that, for ${i \geq 1}$, we have that ${c_{i,r+1}}$ is positive for even~$i$ and negative for odd~$i$. 
		
		\smallskip
		
		\item {\em $m = r+2$:} Using the calculation
		\[\begin{array}{c@{\;}l@{\;}}
			&x^m + \overline x^{\,m}\\[0.35em]
			= & 2 \cdot (x^r + \overline x^{\,r}) - N(x) \cdot (x^{r-1} + \overline x^{\,{r-1}}) \\[0.35em] 
			= & 2 \cdot \left(c_{0,r+1} + c_{1,r+1} N(x) + \ldots + c_{r/2,r+1} N(x)^{r/2}\right) \\[0.15em]  
			 & \;\;- N(x) \cdot \left(c_{0,r} + c_{1, r}N(x) + \ldots + c_{r/2,r} N(x)^{r/2}\right) \\[0.35em] 
			 = &  2c_{0,r+1} + \sum_{i=1}^{r/2} N(x)^i \cdot (2 c_{i,r+1} - c_{i-1,r}) - c_{r/2,r} \cdot N(x)^{r/2+1}, 
		\end{array} \] 
		the proof proceeds in an analogous way.\qedhere
	\end{itemize}
\end{proof}

\end{document}